\documentclass[11pt]{article}

\usepackage[english]{babel}
\usepackage{amsmath}
\usepackage{amssymb,amsfonts,amsthm}
\usepackage{tikz,color,pgf,graphicx,url}
\usetikzlibrary{calc}
\usetikzlibrary{decorations.pathreplacing}
\usepackage{enumerate}     

\usepackage{xcolor} 
\definecolor{mygreen}{RGB}{55, 184, 2}

\usepackage{hyperref}
\hypersetup{
	colorlinks   = true, 
	urlcolor     = black, 
	linkcolor    = black, 
	citecolor   = black 
}
\usepackage[numbers,sort]{natbib}

\usepackage[
top=2.4cm,
bottom=2.4cm,
inner=2.4cm,      
outer=2.4cm,
footskip=40pt     
]{geometry}

\newtheorem{theorem}{Theorem}[section]
\newtheorem{corollary}[theorem]{Corollary}
\newtheorem{lemma}[theorem]{Lemma}

\newtheorem{conjecture}[theorem]{Conjecture}

\newtheorem{definition}[theorem]{Definition}

\newcommand{\ceil}[1]{\left\lceil #1 \right\rceil}
\newcommand{\floor}[1]{\left\lfloor #1 \right\rfloor}

\begin{document}

\title{Cop number of partial cubes}

\author{Nicholas Crawford \footnote{\href{mailto:nicholas.2.crawford@ucdenver.edu}{nicholas.2.crawford@ucdenver.edu}, Mathematical and Statistical Sciences, Universtiy of Colorado Denver, USA}
\and Vesna Ir\v si\v c Chenoweth \footnote{\href{mailto:vesna.irsic@fmf.uni-lj.si}{vesna.irsic@fmf.uni-lj.si}, Faculty of Mathematics and Physics, University of Ljubljana, Slovenia}}

\maketitle

\begin{abstract}
The game of Cops and Robbers on graphs is a well-studied pursuit--evasion model whose central parameter, the cop number, captures the minimum number of pursuers required to guarantee capture of an adversary on a given graph. While the cop number has been determined for many classical graph families, relatively little is known about the important class of partial cubes, i.e., isometric subgraphs of hypercubes.  

In this paper, we establish a lower bound for the cop number of partial cubes and present an upper bound on a subclass of partial cubes. Additionally, we improve these bounds for a particular family of partial cubes: Fibonacci cubes. These graphs are defined as induced subgraphs of hypercubes obtained by forbidding consecutive ones in binary strings. Beyond their natural combinatorial interest, Fibonacci cubes have connections to chemical graph theory, where they serve as models for resonance graphs of certain classes of polycyclic aromatic hydrocarbons. 

\end{abstract}


\noindent
{\bf AMS Subj.\ Class.\ (2020):} 05C57, 05C12

\section{Introduction}
\label{sec:intro}

The pursuit--evasion game \emph{Cops and Robbers} was independently introduced by Quilliot \cite{Quilliot1983} and by Nowakowski \& Winkler \cite{NowakowskiWinkler1983} in the 1980s, with Aigner \& Fromme \cite{AignerFromme1984} later formalizing the notion of \emph{cop number} $c(G)$ of a graph $G$ in 1984.

One of the main research directions in the area is to obtain good general upper bounds for the cop number. The still open conjectures of Meyniel \cite{maamoun-1987} and Schroder \cite{schroeder-2001} consider upper bounds in terms of the order and the genus of a graph, respectively. Another direction receiving a lot of interest is determining the cop number for specific graph families \cite{MR1623932, MR3647488, MR3997485, MR4313187, MR4285899, MR4520060, MR4708898}. Despite substantial progress in studying the cop number on classical graph families such as trees, planar graphs, Cartesian products, and hypercubes, its behavior on the important class of \emph{partial cubes} remains unknown. 

Partial cubes are isometric subgraphs of hypercubes. Among these, \emph{Fibonacci cubes} and \emph{Lucas cubes} are the most popular due to their recursive structures: Fibonacci cubes forbid consecutive 1's in binary strings, and Lucas cubes impose additional cyclic adjacency constraints. For properties of Fibonacci and Lucas cubes, see a recent monograph \cite{book}.

One of the applications of Fibonacci cubes is their relation to chemical graph theory. Fibonacci cubes are known to model \emph{resonance graphs} of linear polycyclic aromatic hydrocarbons, structures where vertices represent perfect matchings and edges are obtained by face rotations, making Fibonacci cubes especially prominent in modeling molecular resonance networks \cite{kz-05}. 
Lucas cubes analogously represent closed-loop or cyclic molecular configurations, further cementing their relevance in molecular graph modeling \cite{zigert-2013, yao-2015}.

Nevertheless, explicit results on the standard cop number of Fibonacci and Lucas cubes, or pursuit--evasion dynamics tailored to their unique structures, have not yet been published. In this paper, we initiate the study of the cop number of partial cubes. We provide a general lower bound for the cop number of partial cubes and present sharpness examples. We also provide a sharp upper bound for the cop number of median graphs, which are an important subclass of partial cubes, and conjecture that the same bound holds for all partial cubes. By applying the above bounds and more elaborate additional techniques, we are able to prove that the cop number of the $n$-dimensional Fibonacci and Lucas cube lies between $\lfloor \frac{n+5}{6} \rfloor$ and $\lceil \frac{n}{4} \rceil$.

\section{Preliminaries}
\label{sec:prelim}

In this section, we present key definitions that will be used throughout the paper. These include graph-theoretic concepts relevant to pursuit-evasion games, graph embeddings, and particular graph families such as Fibonacci and Lucas cubes.

We begin with the central notion in the game of Cops and Robbers.

\begin{definition}[Cop Number]
The \emph{cop number} of a graph $G$, denoted $c(G)$, is the minimum number of cops required to guarantee the capture of a robber in the game of Cops and Robbers played on $G$.

The game proceeds as follows:
\begin{itemize}
    \item The game is played on a finite, simple, undirected graph $G$.
    \item Initially, $k$ cops choose their starting vertices, followed by the robber choosing a starting vertex.
    \item The players move in alternate rounds. In each round, all cops move first, then the robber.
    \item Each player may move to a neighboring vertex or remain in place.
    \item The cops win if any cop occupies the same vertex as the robber. The robber wins if he can evade capture indefinitely.
\end{itemize}
\end{definition}

Recall that the \emph{$n$-dimensional hypercube}, denoted $Q_n$, is the graph whose vertex set consists of all binary strings of length $n$, i.e., $V(Q_n) = \{0,1\}^n$, where two vertices are adjacent if and only if their strings differ in exactly one bit position. 

To define partial cubes and their subclass (median graphs) we need the notion of isometric subgraphs, i.e.\ subgraphs that maintain shortest-path distances. We formally define this as follows. 
Let $G$ be a graph and $H$ an induced subgraph of $G$. The subgraph $H$ is called an \emph{isometric subgraph} of $G$ if, for every pair of vertices $u, v \in V(H)$,
\[
d_H(u, v) = d_G(u, v),
\]
where $d_H$ and $d_G$ denote the shortest-path distances in $H$ and $G$, respectively. 

\begin{definition}[Partial Cube]
    A \emph{partial cube} is a connected isometric subgraph of a hypercube.
\end{definition}

\begin{definition}[Median Graph]
A graph $G$ is a \emph{median graph} if, for every triple of vertices $u, v, w \in V(G)$, there exists a unique vertex $m \in V(G)$, called the \emph{median}, such that:
\[
d(u, m) + d(v, m) = d(u, v), \quad
d(u, m) + d(w, m) = d(u, w), \quad
d(v, m) + d(w, m) = d(v, w).
\]

Equivalently, the vertex $m$ lies on shortest paths between each pair of $u$, $v$, and $w$.
\end{definition}

Recall that every median graph is a partial cube \cite[Theorem 5.75]{ovchinnikov-2011}. 

We next define a notion of subgraph that admits a distance-non-increasing projection from the ambient graph. Let $G$ be a graph and $H$ an induced subgraph of $G$. We say that $H$ is a \emph{retract} of $G$ if there exists a graph homomorphism $r : V(G) \to V(H)$, called a \emph{retraction}, such that:
    $r(v) = v$ for all $v \in V(H)$ (i.e., $r$ fixes $H$ pointwise), and 
    if $uv \in E(G)$, then either $r(u) = r(v)$ or $r(u)r(v) \in E(H)$.

Two of the most famous examples of partial cubes are \emph{Fibonacci cubes} and \emph{Lucas Cubes}. 

\begin{definition}[Fibonacci Cube]
The \emph{$n$-dimensional Fibonacci cube}, denoted $\Gamma_n$, is the subgraph of the $n$-dimensional hypercube $Q_n$ induced by all binary strings of length $n$ that do not contain two consecutive $1's$. Formally,
\[
V(\Gamma_n) = \left\{ x \in \{0,1\}^n : \text{$x$ does not contain the substring ``11''} \right\},
\]
where two vertices are adjacent if their corresponding strings differ in exactly one bit.
\end{definition}

\begin{definition}[Lucas Cube]
The \emph{$n$-dimensional Lucas cube}, denoted $\Lambda_n$, is the subgraph of the $n$-dimensional hypercube $Q_n$ induced by all binary strings $x = x_1x_2\cdots x_n \in \{0,1\}^n$ satisfying:
\begin{itemize}
    \item $x$ does not contain the substring ``11'',
    \item $x_1 = 1$ and $x_n = 1$ do not both hold (i.e., $x$ does not start and end with 1).
\end{itemize}
Vertices are adjacent if their strings differ in exactly one bit.
\end{definition}

We conclude this section with a few standard definitions. Let $G$ be a graph. The \emph{minimum degree} of $G$, denoted $\delta(G)$, is the smallest degree of any vertex in $G$, i.e.,
\[
\delta(G) = \min_{v \in V(G)} \deg(v),
\]
where $\deg(v)$ is the number of neighbors of $v$.

A subset $D \subseteq V(G)$ is a \emph{dominating set} of a graph $G$ if every vertex $v \in V(G) \setminus D$ has at least one neighbor in $D$. That is, for all $v \in V(G)$, either $v \in D$ or there exists $u \in D$ such that $uv \in E(G)$.
The \emph{domination number} of $G$, denoted $\gamma(G)$, is the minimum cardinality of a dominating set in $G$.

\section{Lower bound}
\label{sec:lower}
We begin with a lower bound on the cop number, $c(G)$, in relation to the minimum degree of the graph, $\delta(G)$. We first state a technical lemma that is needed for the main result of this section.

\begin{lemma}
    \label{lem:at most 2 common neighbors}
    If $G$ is a partial cube, $x, y \in V(G)$ and $d(x,y) = 2$, then $x$ and $y$ have at most two common neighbors.
\end{lemma}

\begin{proof}
    Let $x$ and $y$ be as in the statement of the lemma. As they are at distance 2, they differ in exactly two bits, say in bits $i$ and $j$. Thus, every common neighbor of $x$ and $y$ can also differ from $x$ and $y$ only in bits $i$ and $j$. Thus, there are at most two common neighbors, one is obtained by changing bit $i$ in $x$ and the other by changing bit $j$ in $x$.
\end{proof}

\begin{theorem}
    \label{thm:lower}
    If $G$ is a partial cube, then $c(G) \geq \left \lceil \frac{\delta(G)}{2} \right \rceil$.
\end{theorem}

\begin{proof}
    Let $\delta(G) = d$ and consider the game on $G$ with $k = \left \lceil \frac{d}{2} \right \rceil - 1 \leq \frac{d-1}{2}$ cops. We first show that the robber can always select a starting vertex that is at distance at least 2 from each of the cops.

    Let the cops' starting positions be $c_1^0, \ldots, c_k^0$ and let $C = \{ c_1^0, \ldots, c_k^0 \}$. If $C$ is not a dominating set in $G$, then there exists a vertex $x \in V(G)$ such that $d(x, C) \geq 2$ and the robber starts the game in $x$. Otherwise, suppose that $C$ is a dominating set of $G$. As $|C| = k < d+1 \leq |V(G)|$, there is a vertex $u \in V(G) \setminus C$. Let $N(u) \cap C = X$ and let $N(u) \setminus X = Y$. Clearly, $\deg(u) = |X| + |Y| \geq d$. As $G$ is a partial cube, it has no triangles, so there are no edges between $X$ and $Y$. But as $C$ is dominating, every vertex from $Y$ has a neighbor in $C \setminus X$. This neighbor is at distance 2 from $u$ so by Lemma \ref{lem:at most 2 common neighbors}, no three vertices from $Y$ can have a common neighbor in $C \setminus X$, thus $|C| \geq |X| + \ceil{\frac{|Y|}{2}}$. Simplifying this inequality gives
    $$\frac{d-1}{2} \geq |C| \geq |X| + \frac{|Y|}{2} \geq \frac{2 |X| + |Y|}{2} \geq \frac{|X| + d}{2},$$
    which implies $-1 \geq |X|$, a contradiction as $|X| \geq 0$. Thus, we can conclude that $C$ is never a dominating set of $G$ and the robber can always start the game on a vertex that is at distance at least 2 from all cops. 
  
    Suppose that after $m$ rounds, the robber is at a distance of at least 2 from each cop. By Lemma \ref{lem:at most 2 common neighbors}, each cop has at most two common neighbors with the robber, thus at most $2k$ neighbors of the robber are adjacent to some cop. As $2k \leq d-1$, there is at least one vertex that the robber can move to and stay at distance at least 2 from all cops even after they move in the next round.
\end{proof}

This bound cannot be improved in general. If $n$ is odd, $c(Q_n) = \ceil{\frac{n+1}{2}} = \ceil{\frac{\delta(Q_n)}{2}}$. Additionally, if $T$ is a tree, then $c(T) = 1 = \ceil{\frac{\delta(T)}{2}}$.

As $\delta(\Gamma_n) = \delta(\Lambda_n) = \floor{\frac{n+2}{3}}$ \cite{book} we immediately obtain the following.

\begin{corollary}
    \label{cor:fib-lower}
    If $n \geq 1$, then $c(\Gamma_n) \geq \floor{\frac{n+5}{6}}$ and $c(\Lambda_n) \geq \floor{\frac{n+5}{6}}$.
\end{corollary}

\section{Upper Bound}
\label{sec:upper}

For every partial cube (and thus also for every median graph) $G$ there exists the smallest integer $n$ such that $G$ can be embedded into $Q_n$ as an isometric subgraph. Moreover, median graphs are retracts of hypercubes \cite{bandelt-2008}. Thus, as $c(Q_n) \leq \ceil{\frac{n+1}{2}}$ \cite{maamoun-1987} and $c(H) \leq c(G)$ for every retract $H$ of $G$ \cite{berarducci-1993}, we obtain the following.

\begin{corollary}
    \label{cor:hypercube}
    If $G$ is median graph that isometrically embeds into $Q_n$, then $c(G) \leq \ceil{\frac{n+1}{2}}$.
\end{corollary}

As hypercubes are median graphs, the above bound cannot be improved for general median graphs. For Fibonacci cubes, however, we are able to significantly improve this upper bound. It is easy to see that $c(\Gamma_0) = c(\Gamma_1) = c(\Gamma_2) = 1$ and $c(\Gamma_3) = c(\Gamma_4) = c(\Gamma_5) = 2$. See Figure \ref{fig:fib-cubes}. 

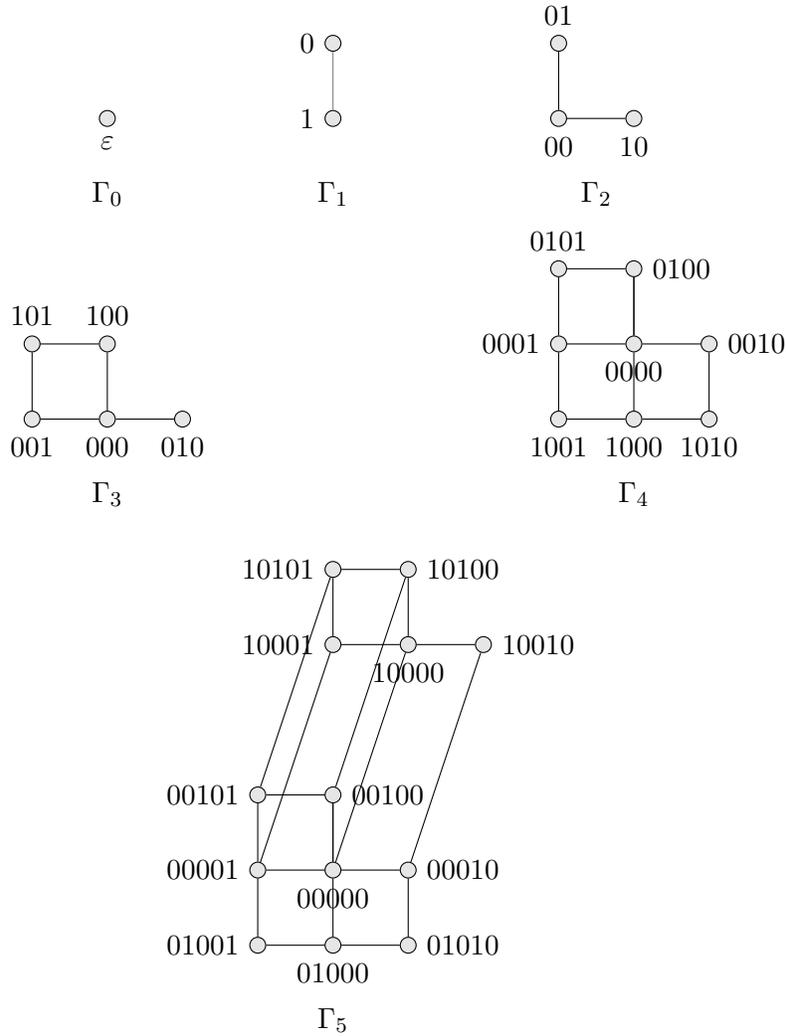
\begin{figure}[ht!]
    \centering
    \begin{tikzpicture}[
        scale=1,
        vert/.style={circle, draw, fill=black!10, inner sep=0pt, minimum width=6pt},
    ]

    \begin{scope}[xshift=0cm, yshift=0cm]
        \node[vert, label=below:$\varepsilon$] (e0) at (0,0) {};
        \node at (0,-1) {$\Gamma_0$};
    \end{scope}

    \begin{scope}[xshift=3cm, yshift=.5cm]
        \node[vert, label=left:$0$] (0) at (0,0.5) {};
        \node[vert, label=left:$1$] (1) at (0,-0.5) {};
        \draw[gray] (0) -- (1);
        \node at (0,-1.5) {$\Gamma_1$};
    \end{scope}

    \begin{scope}[xshift=6cm, yshift=0cm]
        \node[vert, label=below:$00$] (00) at (0,0) {};
        \node[vert, label=above:$01$] (01) at (0,1) {};
        \node[vert, label=below:$10$] (10) at (1,0) {};
        \draw (00) -- (01);
        \draw (00) -- (10);
        \node at (0.5,-1) {$\Gamma_2$};
    \end{scope}

    \begin{scope}[xshift=0cm, yshift=-4cm]
        \node[vert, label=below:$000$] (000) at (0,0) {};
        \node[vert, label=above:$100$] (001) at (0,1) {};
        \node[vert, label=below:$010$] (010) at (1,0) {};
        \node[vert, label=below:$001$] (100) at (-1, 0) {};
        \node[vert, label=above:$101$] (101) at (-1, 1) {};

        \draw (000) -- (001);
        \draw (000) -- (010);
        \draw (000) -- (100);
        \draw (001) -- (101);
        \draw (100) -- (101);

        \node at (0,-1) {$\Gamma_3$};
    \end{scope}

  \begin{scope}[xshift=7cm, yshift=-3cm]
        \node[vert, label=below:$0000$] (0000) at (0,0) {};
        \node[vert, label=below:$1000$] (0001) at (0,-1) {};
        \node[vert, label=right:$0100$] (0010) at (0,1) {};
        \node[vert, label=right:$0010$] (0100) at (1,0) {};
        \node[vert, label=below:$1010$] (0101) at (1,-1) {};
        \node[vert, label=left:$0001$] (1000) at (-1,0) {};
        \node[vert, label=below:$1001$] (1001) at (-1,-1) {}
            edge (1000);
        \node[vert, label=above:$0101$] (1010) at (-1,1) {};

        \draw (0000) -- (0001);
        \draw (0000) -- (0010);
        \draw (0000) -- (1000);
        \draw (0001) -- (1001);
        \draw (0010) -- (0000);
        \draw (0010) -- (1010);
        \draw (0100) -- (0000);
        \draw (0100) -- (0101);
        \draw (0101) -- (0001);
        \draw (1010) -- (1000);

        \node at (0,-2) {$\Gamma_4$};
    \end{scope}

 \begin{scope}[xshift=3cm, yshift=-10cm]
        \node[vert, label=below:$00000$] (0000) at (0,0) {};
        \node[vert, label=below:$01000$] (0001) at (0,-1) {};
        \node[vert, label=right:$00100$] (0010) at (0,1) {};
        \node[vert, label=right:$00010$] (0100) at (1,0) {};
        \node[vert, label=right:$01010$] (0101) at (1,-1) {};
        \node[vert, label=left:$00001$] (1000) at (-1,0) {};
        \node[vert, label=left:$01001$] (1001) at (-1,-1) {};
        \node[vert, label=left:$00101$] (1010) at (-1,1) {};
        
        \node[vert, label=below:$10000$] (000) at (1,3) {};
        \node[vert, label=right:$10100$] (001) at (1,4) {};
        \node[vert, label=right:$10010$] (010) at (2,3) {};
        \node[vert, label=left:$10001$] (100) at (0,3) {};
        \node[vert, label=left:$10101$] (101) at (0,4) {};

        \draw (0000) -- (0001);
        \draw (0000) -- (0010);
        \draw (0000) -- (1000);
        \draw (0001) -- (1001);
        \draw (0010) -- (0000);
        \draw (0010) -- (1010);
        \draw (0100) -- (0000);
        \draw (0100) -- (0101);
        \draw (0101) -- (0001);
        \draw (1010) -- (1000);
        \draw (1001) -- (1000);

        \draw (000) -- (001);
        \draw (000) -- (010);
        \draw (000) -- (100);
        \draw (001) -- (101);
        \draw (100) -- (101);

        \draw (0000) -- (000);
        \draw (0010) -- (001);
        \draw (1000) -- (100);
        \draw (1010) -- (101);
        \draw (0100) -- (010);

        \node at (0,-2) {$\Gamma_5$};        
\end{scope}

    \end{tikzpicture}
    \caption{Fibonacci cubes $\Gamma_n$ for $n \in \{0,1,2,3,4,5\}$. Each vertex is a binary string with no two consecutive $1$'s.}
    \label{fig:fib-cubes}
\end{figure}

\begin{theorem}
    \label{thm:n/3}
    If $n \geq 6$, then $c(\Gamma_n)\leq \lceil \frac{n}{3}\rceil$.
\end{theorem}

\begin{proof}
    Let $k=\lceil \frac{n}{3}\rceil$. We provide a winning strategy for $k$ cops $c_1, \ldots, c_k$. Partition each binary string of length $n$ into $k$ blocks $B_1, \ldots, B_k$ of length 3 except $B_k$ which is of length $n - 3(k-1) \in \{1,2,3\}$. More precisely, the blocks of the string $x_1 \ldots x_n$ are $B_i = x_{3i-2} x_{3i-1} x_{3i}$ for $i \in [k-1]$ and $B_k \in \{x_n, x_{n-1} x_n, x_{n-2} x_{n-1} x_n\}$. In our strategy, which consists of different phases, cop $c_i$ is associated with block $B_i$, and all cops start in $0^n$. If cop $c_i$ and the robber have the same value on $B_j$ we say that $c_i$ matches the robber on $B_j$.

    Before presenting the strategy, observe that vertices with all possible values of $B_i$ where values on other blocks are fixed induce a subgraph of $\Gamma_n$ that is isomorphic to $\Gamma_3$ or an induced subgraph of $\Gamma_3$, and that $000$ is at distance at most 2 from all other possible values of $B_i$ and at distance exactly 2 only from $101$. When describing the cop's strategy to match the robber on some $B_i$ we will say that the cop moves closer to the robber on $B_i$, which means that the cop moves closer to the robber in the corresponding $\Gamma_3$. Saying that a player moves on $B_i$ will mean that they change a bit from $B_i$.

    \begin{description}
        \item[Phase 1.] Until it holds for every $i \in [k]$ except maybe one that $c_i$ matches the robber on $B_i$.\\
        The strategy for cop $c_i$ during Phase 1 is to only change bits in $B_i$, aiming to match the robber there (the remaining bits stay $0$). Once they match, the cop can maintain this property by moving on $B_i$ if and only if the robber moves on $B_i$, and not moving otherwise. If the robber moves on $B_j$, $j \neq i$, then the cop $c_i$ moves closer to the robber on $B_i$ (decreasing the distance between them by 1). If the robber moves on $B_i$, $c_i$ moves such that the distance between the cop and the robber does not increase. If the robber does not move on $B_i$ for at least two rounds, $c_i$ can match the robber on $B_i$. Since $k \geq 3$, there is at most one block $B_i$ on which $c_i$ is not able to match the robber. Without loss of generality, let this be the block $B_k$.

        \item[Phase 2.] For every $i \in [k]$, the cop $c_i$ matches the robber on $B_i$.
        \begin{description}
            \item[Phase 2.1.] The game never reaches the condition from Phase 2.\\
            By the strategy from Phase 1, this is only possible if the robber always moves on $B_k$ (except maybe once). Thus, cops $c_1, \ldots, c_{k-1}$ can match the robber on $B_1$ by moving closer to the robber there. Then they can also match the robber on $B_2$, etc., until they all match the robber on blocks $B_1, \ldots, B_{k-1}$. Note that by matching the robber on blocks sequentially, we avoid the possible problem of the cop wanting to move to a vertex outside of $V(\Gamma_n)$. As $k \geq 3$, we thus have at least two cops that match the robber everywhere except on $B_k$. The final part of their strategy is to imagine the game on $\Gamma_{|B_k|}$ according to the robber's moves on $B_k$ and as $c(\Gamma_m) \leq 2$ for $m \in [3]$, these two cops can win the game. Once they do, the robber is caught on $\Gamma_n$ too.

            Note that during this phase, cop $c_i$ always matches the robber on $B_i$ for every $i \in [k-1]$. Observe also that if on some $B_j$ cop $c_i$ is at distance 2 from the robber anytime during Phase 2.1, then $c_i$ is on $000$ and the robber is on $101$. 

            \item[Phase 2.2.] After the condition from Phase 2 is established.\\
            Suppose the robber moves on $B_j$. If the cop $c_i$ matches the robber on $B_j$, then $c_i$ simply moves on $B_j$ as well, so that he still matches the robber. If $c_i$ does not yet match the robber on $B_j$, then we distinguish between the following cases.
            \begin{itemize}
                \item If robber is at distance 1 from $c_i$ on $B_j$, $c_i$ moves such that he matches the robber on $B_j$ after this move.
                \item If robber is at distance $2$ from $c_i$ on $B_j$, then as observed in Phase 2.1., $c_i$ is on $000$ and robber is on $101$. Cop's strategy is to move to $001$. As on $B_{j+1}$, $c_i$ either matches the robber (who is on $101$ on $B_j$) or is on $000$, this is a legal move in $V(\Gamma_n)$. If in the next round, the robber moves on $B_j$ again and does not match $c_i$, then $c_i$ returns to $000$. (To avoid forcing the cop to move outside of $V(\Gamma_n)$ when trying to match the robber on $B_{j+1}$.) Otherwise, if the robber moves on some $B_\ell$, $\ell \neq j$, then $c_i$ can match the robber on $B_j$ after his move. 
            \end{itemize}
            This strategy ensures that if the robber moves on some $B_j$ and then moves on a different block in the next round (or does not move at all), all cops match the robber on $B_j$ after these two rounds. So, unless the robber is moving only on one block for the rest of the game, he is caught in a finite number of rounds. If the robber keeps moving on only one block, say on $B_\ell$, then first wait for all cops except $c_\ell$ to match the robber on blocks $B_i$, $i \in [k] \setminus \{\ell\}$. Afterwards, as $\Gamma_3$ has only one vertex at a distance 2 from $000$ and the robber is always moving on $B_\ell$, either in this or the next round, the robber is on $001$ or $100$, so all cops but $c_\ell$ now also match the robber on $B_\ell$, thus, the robber is caught. \hfill \qedhere
        \end{description}
    \end{description}
\end{proof}

Using the same methods, we obtain an even stronger upper bound. Note that as the diameter of $\Gamma_5$ is 3, this method cannot be used further.

\begin{theorem}
    \label{thm:n/4}
    If $n \geq 9$, then $c(\Gamma_n)\leq \lceil \frac{n}{4}\rceil$.
\end{theorem}

\begin{proof}
    The idea of the proof is the same as to prove Theorem \ref{thm:n/3}. Now we play with $k = \lceil \frac{n}{4}\rceil$ cops, and each binary string of length $n$ is divided into $k$ blocks $B_1, \ldots, B_k$, each of length 4 (except possibly $B_k$, which could be shorter). 

    Now the vertices with all possible values on $B_i$ but with other values fixed induce $\Gamma_4$ (or an induced subgraph of $\Gamma_4$ in the case $i=k$). The vertex $0000$ is at distance at most 2 from all other vertices of $\Gamma_4$ and at distance exactly 2 only from $0101, 1001, 1010$. 

    Phase 1 is the same as in the proof of Theorem \ref{thm:n/3}. In Phase 2.1, the only difference is that $c(\Gamma_m) \leq 2$ for all $m \in [4]$. Additionally, during this phase, when trying to match the robber on some $B_i$, the cops always return back to $0000$ if they are not able to match the robber within two rounds. Thus, if on some $B_j$ cop $c_i$ is at distance 2 from the robber, then $c_i$ is on $0000$.

    In Phase 2.2, if the robber is at distance 2 from $c_i$ on $B_j$, then $c_i$ moves to $0001$ if the robber is on $0101$ or $1001$, and to $1000$ if the robber is on $1010$. If he cannot match the robber in the next round, he returns to $0000$. Unless the robber keeps moving only on one block for the rest of the game, he is again caught. Otherwise, the robber is always moving on some $B_\ell$, thus at least every second round he is at a distance at most 1 from $0000$, so cops can win again.
\end{proof}

An important note is that Lucas cubes are retractions of Fibonacci cubes \cite{klavzar-2005} and as a consequence the upper bounds in Theorems \ref{thm:n/3} and \ref{thm:n/4} for Fibonacci cubes are also upper bounds for Lucas cubes. 

\begin{corollary}
    \label{cor:lucas}
    If $n \geq 9$, then $c(\Lambda_n)\leq \lceil \frac{n}{4}\rceil$.
\end{corollary}

It can also be checked that $c(\Lambda_2) = c(\Lambda_3) = 2$, $c(\Lambda_4) = c(\Lambda_5) = c(\Lambda_6) = c(\Lambda_7) = 2$ and $c(\Lambda_8) = 3$.

\section{Further directions}

In this paper, we prove that the cop number of a partial cube $G$ is bounded from below by $\lceil \frac{\delta(G)}{2} \rceil$ and that the cop number of a median graph embedded in $Q_n$ is bounded from above by $\lceil \frac{n+1}{2} \rceil$. However, we believe this result generalizes to all partial cubes.

\begin{conjecture}
    If $G$ is a partial cube that isometrically embeds into $Q_n$, then $c(G) \leq \left \lceil \frac{n+1}{2} \right \rceil$.
\end{conjecture}

However, even if the above conjecture holds, there is a gap between the lower and the upper bounds. Thus, it would be interesting to determine exact values or at least improve these bounds for famous partial cubes, for example, daisy cubes, Pell graphs, generalized Pell graphs, metallic cubes, Horadam cubes, Fibonacci $p$-cubes, and weighted Padovan graphs. As already shown in the case of Fibonacci cubes, this general upper bound can at least sometimes be improved. Note, however, that the same technique used in Theorems \ref{thm:n/3} and \ref{thm:n/4} cannot be applied to show that $c(\Gamma_n) \leq \lceil \frac{n}{5} \rceil$ as $\Gamma_5$ does not have the required property (a vertex $x \in V(\Gamma_5)$ such that $N_2[x] = V(\Gamma_5)$).

\section*{Acknowledgments}

V.I.C.\ acknowledges the financial support from the Slovenian Research and Innovation Agency (Z1-50003, P1-0297, N1-0218, N1-0285, N1-0355) and the European Union (ERC, KARST, 101071836).

\end{document}